\documentclass[12pt]{article}
\usepackage[T1]{fontenc}
\usepackage[english]{babel}

\usepackage[a4paper]{geometry}
\usepackage{amsmath, amssymb, amsthm}
\usepackage{mathtools}
\usepackage{graphicx}
\usepackage{xcolor}
\usepackage[colorlinks=true, allcolors=blue]{hyperref}
\usepackage{algorithm}
\usepackage{algpseudocode}
\usepackage{apptools} 

\usepackage[backend=biber,style=alphabetic,sorting=ynt]{biblatex}
\AtBeginBibliography{\small}

\usepackage{tikz}
\usetikzlibrary{shapes.geometric, 
                arrows,
                positioning,
                calc}
\tikzstyle{block} = [rectangle, 
                    rounded corners,
                    minimum width=7cm,
                    minimum height=1.4cm,
                    align=center,
                    draw=black,
                    fill=teal!30]
\tikzstyle{arrow} = [ultra thick,->,>=latex]
\newcommand{\R}{\mathbb{R}}
\newcommand{\Z}{\mathbb {Z}}
\newcommand{\calP}{{\mathcal P}}
\DeclareMathOperator*{\argmin}{arg\,min}
\DeclareMathOperator{\supp}{supp}
\newcommand{\mubar}{\bar{\mu}}
\newtheorem{thm}{Theorem}
\newtheorem*{thm*}{Theorem}
\newtheorem{prp}[thm]{Proposition}
\newtheorem{lma}[thm]{Lemma}

\theoremstyle{definition}

\AtAppendix{\counterwithin{thm}{section}}

\title{The GenCol 
algorithm for high-dimensional optimal transport: general formulation and application to barycenters and Wasserstein splines 
}
\author{
    Gero Friesecke\thanks{Department of Mathematics, Technische Universität München}\\
    \nolinkurl{gf@ma.tum.de}
    \and
    Maximilian Penka\footnotemark[1]\\
    \nolinkurl{penka@ma.tum.de}
}

\begin{document}
\maketitle

\begin{abstract}\footnotesize
\noindent We extend the recently introduced genetic column generation algorithm for high-dimensional multi-marginal optimal transport from symmetric to general problems. We use the algorithm to calculate accurate mesh-free Wasserstein barycenters and cubic Wasserstein splines.
\end{abstract}

\section{Introduction}
Solving multi-marginal optimal transport problems is an important task in many applications; see e.g. \cite{pass-15, Nen-PhD-16, FriGorGer-22}. Current algorithms for optimal transport suffer from the curse of dimension when transferred from the two-marginal to the multi-marginal setting. The curse not only concerns computation speed, but also memory demand to even just store any transportation plan. The principal idea of the Genetic Column Generation (GenCol) algorithm, introduced in \cite{FriSchVoe-21} for symmetric problems arising in electronic structure and extended here to general optimal transport problems, is to keep and exploit, rather than smear out, the extreme sparsity of exact optimizers. 

To explain this idea in more detail, we begin by stating the general multi-marginal optimal transport (MMOT) problem: 
given $N$ probability measures $\mu_1$, ..., $\mu_N$ on, say, a region $X \subset \R^d$, 
\begin{align} 
 &  \mbox{ minimize } \int_{X^N} c(x_1,\dots,x_N) \, d\gamma(x_1,\dots,x_N)\mbox{ over probability measures }\gamma \mbox{ on }X^N \nonumber \\
 &  \mbox{ subject to }       \int_{\pi_k^{-1}(A)}d\gamma = \int_Ad\mu_k \quad \forall \mbox{ measurable } A\subset X, \forall k = 1,...,N. \label{MMOT-pb}   
\end{align}
Here $\pi_k\, :\, X^N \to X$ is the projection on the k-th component and $c \, : \, X^N \to \R$ is a given cost function. Validity of the marginal constraints in \eqref{MMOT-pb} is denoted $\gamma\mapsto\mu_1,...,\mu_N$, and optimizers are called optimal plans.

If all marginals $\mu_k$ are discrete measures -- e.g. after discretization -- the problem becomes a linear program (LP). But the number of unknowns grows exponentially with the number of marginals, 
putting many applications out of reach. 
Already $N=5$ two-dimensional marginals (say, images) on a 50$\times$50 grid result in an intractable problem size of $10^{17}$ variables.  

The currently widely used Sinkhorn or IBP algorithm \cite{cuturi2014fast,BenEtAl-15}, while very successful for two-marginal problems, is not well suited for the computation of multi-marginal problems beyond a small number of marginals. 
%
This is because the utilized smoothing makes the optimal plan positive everywhere, prohibiting already its storage. Setting values below a threshold to zero (as is sometimes done in Sinkhorn computations) does not cure the problem. The optimal plan would stay positive in a small-diameter region around the support; but the curse of dimension already occurs locally in such a region. For 10 two-dimensional marginals (images) on a 50$\times$50 grid, a small ball of the radius of 5 gridpoints around a single support point already contains more than $10^{12}$ gridpoints.

By contrast, 
we show below (see section \ref{sec:discrete}) that the discrete MMOT problem \eqref{MMOT-pb} possesses 
extremely sparse solutions. If the marginals $\mu_k$ are supported on $\ell_k$ points, regardless of the cost function there exist optimal plans with support size less than 
the {\it sum} of the $\ell_k$, instead of the {\it product} of the $\ell_k$ needed for general plans. For 10 marginals on a 50$\times$50 grid, this amounts to a support size of just 25$\,$000.  
The GenCol algorithm for optimal transport, introduced in \cite{FriSchVoe-21} in the context of symmetric problems arising in electronic structure, is designed to preserve  this sparsity and make large problems accessible.

In this paper we generalize the algorithm to general multi-marginal optimal transport problems.

In short, GenCol works as follows. One iteratively solves only sparse LPs obtained by restricting the MMOT problem to a small subset $\Omega$ of $X^N$ whose size is $\beta$ times the support size of optimal plans, where $\beta$ is a hyperparameter of the method (taken to be 3 in all our numerical examples), and updates $\Omega$ by a simple but powerful genetic learning method. An indication of problem sizes that now become accessible is given in Figure \ref{fig:size}. 

\begin{figure} 
\includegraphics[width=0.9\textwidth]{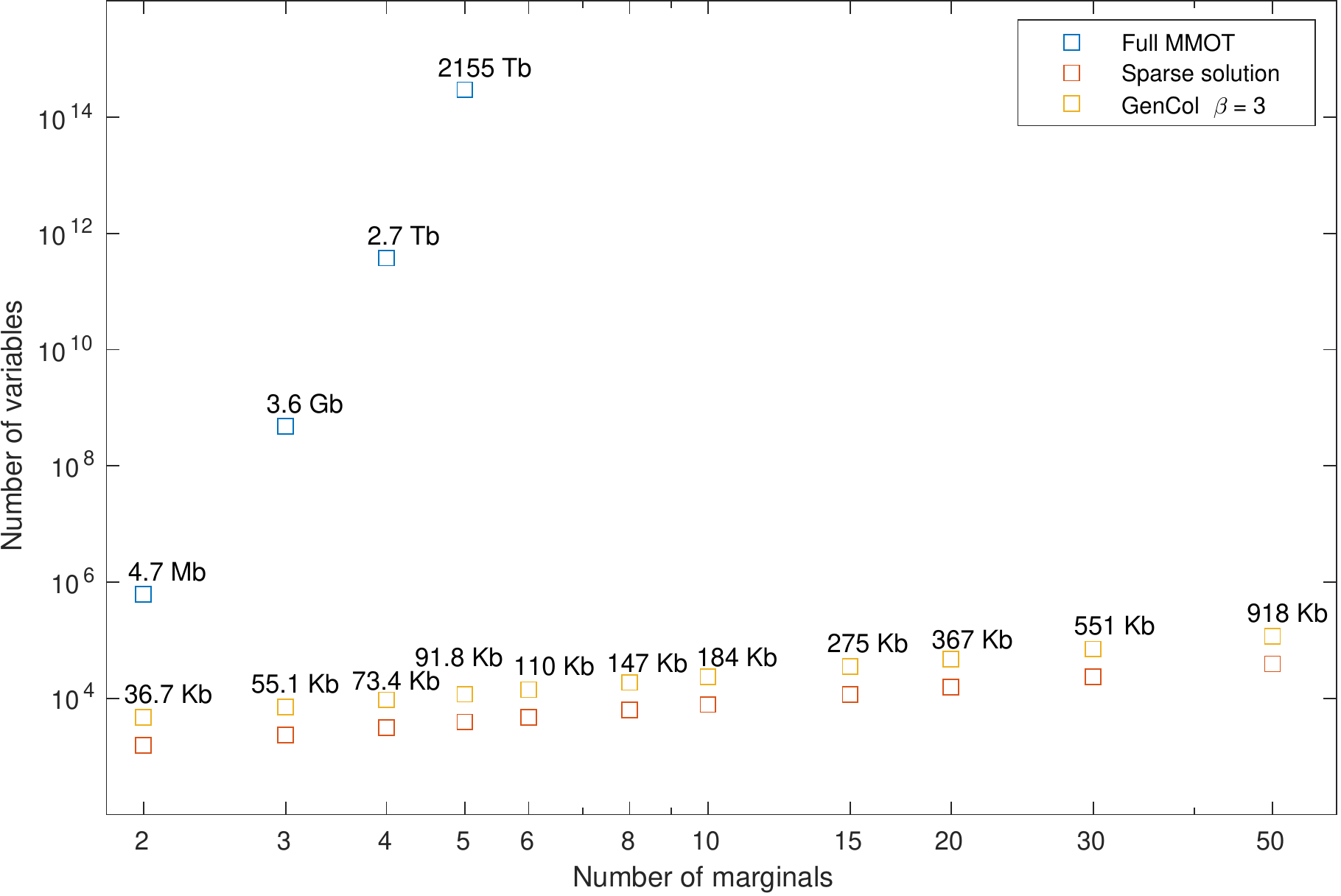}
\caption{Number of variables and storage cost in double precision for a full MMOT plan, a sparse exact solution, and a sparse temporary solution in GenCol. Here we considered the MNIST dataset where each marginal is a 28$\times$28 image. The cost for a full plan grows exponentially with the number of marginals, and already exceeds 2 Terabyte for only 4 marginals. By contrast, the cost in GenCol grows only linearly, making many-marginal problems accessible.}

\label{fig:size}
\end{figure}

 As a proof of concept we use the GenCol algorithm to (1) compute the optimal transport plans in the multi-marginal formulation of the Wasserstein barycenter problem to produce the mesh-free Wasserstein barycenter of MNIST digits and obtain accurate morphed shapes, (2) calculate cubic spline interpolations with respect to the Wasserstein distance.

In our first example, the Wasserstein barycenter problem introduced in the celebrated paper \cite{agueh2011barycenters}, the marginal measures $\mu_1,...,\mu_N$ represent given datasets and the cost function $c(x_1,...,x_N)$ in \eqref{MMOT-pb} is the {\it mean squared deviation from the classical barycenter} of the $x_i$, 
\begin{equation} \label{barycost}
  c(x_1,...,x_N) =  \sum_{i=1}^N\lambda_i|x_i-B_\lambda(x_1,...,x_N)|^2 
  \mbox{ with }
    B_\lambda(x_1,...,x_N) := \sum_{i=1}^N \lambda_i x_i.
\end{equation}
Here the $\lambda_i$ are given positive weights with $\sum_{i=1}^N\lambda_i = 1$. From an optimal solution $\gamma^*$ to this MMOT-problem, the Wasserstein barycenter of the measures $\mu_i$ is then obtained as the push-forward under the barycenter map, 
\begin{align} \label{pfw}
    \mubar = (B_\lambda)_\#\gamma^*.
\end{align}
For this particular problem, there is an alternative ``coupled two-marginal'' formulation, 
\begin{equation}\label{C2M}
  \mubar =  \argmin_{\mu \in \mathcal{P}(X)} \sum_{k = 1}^N\lambda_k W_2(\mu,\mu_k)^2,
\end{equation}
where $W_2$ is the Wasserstein-2 distance. On $X=\R^d$ these formulations are equivalent \cite{agueh2011barycenters},  
in the sense that $\mubar$ is a minimizer of \eqref{C2M} if and only if it is of the form \eqref{pfw} for some minimizer of \eqref{MMOT-pb}, \eqref{barycost}. The coupled two-marginal formulation can be efficiently simulated with the Sinkhorn algorithm  \cite{cuturi2014fast}, see also 
\cite{cuturi2016smoothed,schmitzer2019stabilized,PeyCut-19}. Other recent approaches to compute the coupled two marginal formulation can be found in \cite{yang2021fast, heinemann2022randomized}.
The multi-marginal formulation has to our knowledge not previously been simulated except for 3 marginals \cite{BenEtAl-15}, but we argue that doing so provides a worthwhile alternative method which is very accurate and achieves finer resolution. The latter point is related to the interesting fact that {\it after discretization the two formulations are no longer equivalent!} To see this, suppose the $\mu_k$ are discrete measures on a 2D grid $X_h=[0,1]^2\cap h\Z^2$ of meshsize $h$. Replacing $X$ by $X_h$ in \eqref{pfw} and \eqref{MMOT-pb}, the coupled two-marginal solution is again a measure on $X_h$, whereas the multi-marginal plan $\gamma^*$ is a measure on the $N$-fold cartesian product $(X_h)^N$ and thus its push-forward \eqref{pfw} (say, for $\lambda_1=...=\lambda_N=\tfrac{1}{N}$) lives on the $N^2$ times finer grid $[0,1]\cap (\tfrac{h}{N}\Z)^2$, and is hence more accurate. In fact, interpreting the discrete marginals $\mu_k$ and the multi-marginal barycenter $\mubar$ as sums of Dirac measures on $\R^2$ positioned at the gridpoints, $\mubar$ is the true (mesh-free) barycenter in $\calP(\R^2)$ of the marginals. For MNIST digits, the higher accuracy of multi-marginal barycenters computed with the GenCol algorithm is clearly visible in Figure \ref{fig:MNISTallBC}.


%
%
Our second example, the construction of interpolating curves of higher smoothness in Wasserstein space, is a very recent development in optimal transport \cite{benamou2019second, chen2018measure}. One wants to interpolate given measures $\mu_{t_i}$, $0=t_0< ... < t_N =1$, by a smooth path of measures. This corresponds to a multi-marginal problem \eqref{MMOT-pb}, with a prototypical cost function given by  the cubic spline energy 
\begin{equation} \label{splinecost}
    c(x_0,...,x_N) = \min_{{\rm paths }\, x \, : \, [0,1]\to X} \left\{\int_0^1 |\ddot{x}(t)|^2 dt \, \Big| \, x\in C^2, \, x(t_i)=x_i \, \forall i=0,...,N \right\}.
\end{equation}
(An explicit expression in terms of the $x_i$ and $t_i$ can be found in Appendix \ref{apdx:CubicSplineCost}; for equidistant time steps $\tau$ a good approximation is  $\tilde{c}(x_0,...,.x_N)=\sum_{i=1}^{N-1}|x_{i+1}-2x_i+x_{i-1}|^2/\tau^3$.) The interpolating measures $\mu_t$, $t\in[0,1]$, are then obtained by 
\begin{equation} \label{interpol}
   \mu_t = (E_t)_\sharp \gamma^*,
\end{equation}
where $\gamma^*$ solves \eqref{MMOT-pb}, \eqref{splinecost} and $E_t(x_0,...,x_N)=x(t)$ is the value of the optimal path in \eqref{splinecost} at time $t$. This problem cannot be reduced to two-marginal problems and ours appears to be the first method which can accurately and efficiently solve the governing Kantorovich problem for general data.

\section{Discretization of the multi-marginal problem} \label{sec:discrete}

After discretization, marginals on a continuous state space (typically, a region $X$ of $\R^d$) become discrete probability measures on a finite set of gridpoints. For the multi-marginal problem \eqref{MMOT-pb}, we allow the set of discretization points for the $k$-th marginal to be $k$-dependent, 
$$
    X_k = \{ a_1^{(k)},...,a_{\ell_k}^{(k)}\},
$$
so the $k$-th marginal becomes a discrete probability measure on $X_k$. Multi-marginal plans then become discrete probability measures on the product grid $X_1\times ... \times X_N$. In the following, we identify these probability measures with their densities with respect to counting measure. That is, 
\begin{itemize}
\item the $k$-th marginal is a positive function $\mu_k$ on $X_k$ (or a positive vector in $\R^{\ell_k}$) with $\sum_{r_k\in X_k}\mu_k(r_k)=1$
\item transport plans are nonnegative functions $\gamma$ on $X_1\times ... \times X_N$ (or $\ell_1\times ... \times \ell_N$ tensors in $\R^{\ell_1+...+\ell_N}$) with $\sum_{r=(r_1,...,r_N)\in X_1\times ... \times X_N} \gamma(r)=1$
\item the cost $c$ also becomes a function on $X_1\times ... \times X_N$ (or an $\ell_1\times ... \times \ell_N$ tensor).
\end{itemize}
%


\noindent
The MMOT problem turns into the linear program
\begin{align} \label{discrete-MMOT-pb} \tag{MMOT}
    \operatorname{minimize} \quad&  \langle c, \gamma\rangle \mbox{ over }\gamma\, : \, X_1\times ... \times X_N\to\R \\
    \text{subject to}       \quad& M_k\gamma = \mu_k \quad \forall k = 1,...,N \notag\\
                            & \gamma \geq 0, \notag   
\end{align}
where
\begin{itemize}
    \item  $\langle c,\gamma\rangle := \sum_{\mathrm{r} \in X_1\times\dots\times X_N}c(\mathrm{r})\gamma(\mathrm{r})$ is the (Frobenius) inner product
    \item $M_k: \R^{\ell_1 \times \dots \times \ell_N} \to \R^{\ell_k}$ is the marginal operator:
    \[
      (M_k\gamma)(r_k) := \sum_{\substack{(r_1,...,r_{k-1},r_{k+1},...,r_N) \\ \in X_1\times ... \times X_{k-1} \times X_{k+1}\times ... \times X_N}} \gamma(r_1,\dots,r_N) \; \forall r_k\in X_k.
    \]
\end{itemize}
The set of transport plans satisfying the constraints in \eqref{discrete-MMOT-pb}, 
\begin{equation} \label{Kantpoly}
   \Pi(\mu_1,...,\mu_N) := \{\gamma \, : \, X_1\times ... \times X_N \to \R\, | \, M_k\gamma = \mu_k \forall k=1,...,N, \, \gamma\ge 0\},
\end{equation}
is a convex polytope known as the  Kantorovich polytope. 

%
%

%
%
\section{Sparsity} \label{sec:sparse}
Starting point of the algorithm introduced shortly is the fundamental fact that optimal transport problems admit sparse optimizers. The ancestor of all such results is the celebrated Brenier's theorem \cite{brenier1991polar} which states that solutions to the two-marginal problem with quadratic cost,
$$
\mbox{minimize }\int_{\R^d\times\R^d} |x-y|^2 d\gamma(x,y) \mbox{ subject to }\gamma\mapsto\mu_1,\mu_2
$$
with absolutely continuous first marginal $\mu_1$ are supported on the graph of a  map,
$$
  d\gamma(x,y)=d\mu_1(x)\delta_{T(x)}(y)\, dy \mbox{ for some }T\, : \, \R^d\to\R^d.
$$
Thus the support of the optimal plan, being contained in the set   $\{(x,y)\in\R^{2d}\, | \, y=T(x)\}$, is locally only $d$-dimensional instead of $2d$-dimensional. An analogous Monge ansatz for multi-marginal problems, 
\begin{equation} \label{MM-Monge}
  d\gamma(x_1,...,x_N)=d\mu_1(x_1)\delta_{T_2(x_1)}(x_2) \cdots \delta_{T_N(x_1)}(x_N) \, dx_2 ... dx_N,
\end{equation}
has been justified for certain special cost functions  \cite{GanSwi-98, agueh2011barycenters, ColEtAl-2015}. When valid, such results express an even more remarkable sparsity of optimizers than Brenier's theorem: optimizers are supported on the graph of the map $(T_2,...,T_N)\, : \, \R^d\to\R^{(N-1)d}$, yielding a local support dimension of $d$ instead of $Nd$. While the ansatz \eqref{MM-Monge} fails in simple counterexamples (see e.g. \cite{Fri-19,GerKauRaj-19}), in the case of {\it discrete and symmetric} multi-marginal problems a suitable sparse ansatz valid for general costs and marginals was found in \cite{friesecke2018breaking}. Here we extend this result to general multi-marginal problems, and provide a different proof which relies on convex geometry applied to the Kantorovich polytope instead of linear algebra applied to the constraint matrix. 
\begin{thm}\label{thm:SparseOptimizers}
Let $\mu_1,...\mu_N$ be probability vectors of length $\ell_k$ and let $C\in\R^{\ell_1\times ... \times \ell_N}$ be any cost tensor. Then the problem
\begin{equation} \label{apdx:MMOT}
\begin{aligned} 
				&\textit{minimize } \langle C, \gamma \rangle \mbox{ over }\gamma\in\R^{\ell_1}\times ... \times \R^{\ell_N} \\ &\textit{subject to } \gamma \in \Pi(\mu_1,\dots,\mu_N)
	\end{aligned}	
\end{equation} 
(see \eqref{Kantpoly}) possesses an optimizer $\gamma^\star$ with at most $\sum_{k=1}^N(\ell_k-1) + 1$ non-zero coefficients.
\end{thm} 
Taking a convex-geometric viewpoint, this theorem can be seen to be a direct consequence of the following general result by Dubins \cite{dubins-62}: 
\begin{thm*}[Dubins]
	Let $L$ be the intersection of a closed and bounded convex set $K\subset\R^d$ with $n$ hyperplanes. Then every extreme point of $L$ is a convex combination of at most $n + 1$ extreme points of $K$.
\end{thm*}
(Recall that an extreme point of a convex set is any point in the set  that cannot be written as a convex combination of two other points in the set.) The proof of this theorem is somewhat complicated, owing to the fact that it treats general compact convex sets. Hence in an appendix we supply a simple proof in the special case of convex polytopes, which is sufficient for our purposes. 

\begin{proof}[Proof of Theorem \ref{thm:SparseOptimizers}] 1. We begin by counting the number of constraints which the marginal conditions impose on probability measures in the set
$$
 \calP_{\ell_1,...,\ell_N}=\{\gamma\in\R^{\ell_1\times ...\ell_N}\, | \, \gamma\ge 0, \, \sum_{i_1,...,i_N}\gamma_{i_1...i_N} = 1\}.
$$
Each marginal condition $M_k\gamma=\mu_k$ imposes $\ell_k-1$ conditions, because it corresponds to $\ell_k$ linear equations one of which is redundant due to the normalization condition  $\sum_{i_1,...,i_N}\gamma_{i_1...i_N}=1$. Hence in total there are $n_*:=\sum_{k=1}^{N} (\ell_k-1)$ constraints. 

2. We now apply Dubins' theorem to the Kantorovich polytope $L=\Pi(\mu_1,...,\mu_N)$, taking 
$K$ to be the probability simplex $\calP_{\ell_1,...,\ell_N}$. Thus $L$ is obtained from $K$ by $n_*$ constraints, each of which corresponds geometrically to intersecting $K$ with some hyperplane. It is obvious that the extreme points of $K$ are precisely the tensors with only 1 nonzero component, corresonding to the Dirac measures on $X_1\times ...\times X_N$. By Dubins' theorem, the extreme points of $L$ are convex combinations of at most $n_*+1$ extreme points of $K$, and thus contain at most $n_*+1$ nonzero components.

3. The assertion now follows from the general fact that every linear function on a compact convex set attains its minimum at some extreme point of the set.
\end{proof}
\section{The GenCol algorithm}
Existing computational methods have their merits, but suffer from the curse of dimension when applied to high-dimensional multi-marginal problems. In particular, the widely used Sinkhorn algorithm abandons the sparsity of exact solutions from Theorem \ref{thm:SparseOptimizers} by smoothing out their support, the smoothing being essential for algorithm convergence. 

Recently, in the special case of symmetric problems arising in electronic structure an algorithm called GenCol was introduced \cite{FriSchVoe-21} which instead {\it preserves} this sparsity.  This makes the algorithm applicable to very high-dimensional problems where smoothed-out plans can no longer be stored or efficiently manipulated. Here we show that the symmetric structure exploited in \cite{FriSchVoe-21} is in fact not needed and the algorithm can be extended to general multi-marginal optimal transport problems.

The algorithm alternates between solving the MMOT problem on a small reduced set $\Omega\subset X_1\times ... \times X_N$ of configurations with $|\Omega|=O(\ell_1+...+\ell_N)$, and updating the reduced set $\Omega$ based on the (primal and dual) MMOT solution until the cost has converged. Crucially, by Theorem \ref{thm:SparseOptimizers} the above size of $\Omega$ is sufficient to solve the MMOT problem exactly. Let us describe each step in turn. At the end of the section we will comment on differences to the symmetric case. \vspace*{2mm}

{\bf The reduced MMOT problem.} 
In the contexts of the general MMOT problem \eqref{discrete-MMOT-pb}, the {\it reduced MMOT problem}  on $\Omega\subset X_1\times ... \times X_N$ is the following:  
\begin{align}\tag{RMMOT}\label{eq:RMMOT}
\begin{split}
    \operatorname{minimize} \quad& \langle c, \gamma\rangle_\Omega \mbox{ over }\gamma\, : \,\Omega\to\R \\
    \text{subject to}       \quad& M_k\gamma = \mu_k  \text{ for all } k=1,...,N, \\
                            & \gamma \geq 0.   
\end{split}
\end{align}
Here $\langle c,\gamma \rangle_\Omega = \sum_{r \in \Omega} c(r)\gamma(r)$ is the (Frobenius) inner product on $\Omega$. As opposed to \eqref{discrete-MMOT-pb}, this is just a small LP; the constraint matrix is of size $(\ell_1+...+\ell_N)\times|\Omega|$. 

The subset $\Omega\subset X_1\times ... \times X_N$ must be such that the constraints are feasible, so that the reduced problem possesses a solution. As we shall see, the updating rule for $\Omega$ automatically preserves this property. It remains to find a feasible initial choice. This problem can be completely solved by using a multi-marginal generalization of the north-west corner rule, as shown in the next section.  
In the following, an important role will be played by the dual of \eqref{eq:RMMOT}, which is
\begin{align}\tag{DRMMOT}\label{eq:DRMMOT}
\begin{split}
    \operatorname{maximize} \quad&  \sum_{i=1}^N\langle p_i, u_i\rangle =: J(u) \mbox{ over }u=(u_1,...,u_N), \, u_i \, : \, X_i\to\R \\
    \text{subject to}       \quad& \sum_{i=1}^N u_i(r_i) \leq c(r) \text{ for all }r \in \Omega .
\end{split}
\end{align}
The solutions $u_1,...,u_N$ are called Kantorovich potentials. 
By LP duality, $\max_u J(u)=\min_\gamma \langle c,\gamma\rangle_\Omega$. 
\vspace*{2mm}

{\bf Updating the reduced configuration set.} The primal problem \eqref{eq:RMMOT} lives on a reduced domain $\Omega\subset X_1\times ... \times X_N$ but considers all constraints. By contrast, the dual problem, \eqref{eq:DRMMOT}, lives on the full domain (i.e. the $u_i$ are functions on $X_i$) but only considers a reduced set of constraints. Starting point for the updating rule is the following well known lemma from discrete optimization which we translate here into the context and language of MMOT.

\begin{lma} \label{lem:justify_update}
    Let $\gamma^*$, $u^*=(u_1^*,...,u_N^*)$ be optimal solutions for the reduced problems \eqref{eq:RMMOT}, \eqref{eq:DRMMOT}. If $u^*$ is feasible for the dual of the full MMOT problem (eq.~\eqref{eq:DRMMOT} with $\Omega$ replaced by $X_1\times ...\times X_N$), i.e. satisfies 
    $$ \sum_{i=1}^Nu_i^\star(r_i) - c(r) \leq 0 \text{ for all } r \in X_1\times\dots\times X_N, $$
    then  
    $$\bar \gamma^\star(r) := \begin{cases} \gamma^\star(r) &  r\in \Omega,\\ 0 & r \not\in \Omega \end{cases}$$ 
    is optimal for the full MMOT problem \eqref{discrete-MMOT-pb}.
\end{lma}
A short proof is included at the end of this section. Now suppose $\gamma$ and $u$ are solutions to \eqref{eq:RMMOT}, \eqref{eq:DRMMOT}. The lemma says that unless the optimizer $\gamma$ of the reduced problem (extended by zero) is already optimal for the full problem, there exists some configuration $r'\in X\backslash\Omega$ where the optimizer $u$ of the dual reduced problem \eqref{eq:DRMMOT} violates the dual constraint, i.e. satisfies  
\begin{equation} \label{acc_crit}
    \sum_{i=1}^N u_i(r_i') - c(r') > 0.
\end{equation}
This fact is exploited by the updating rule for $\Omega$, which goes as follows: 
\begin{itemize}
\item[--] pick a random ``parent'' configuration $r$ with $\gamma(r)>0$ (i.e. $r\in\supp\gamma$) 
\item[--] pick a random ``child'' $r'$ differing from $r$ by only one entry \\(i.e. $r_i'=r_i$ for $i\in\{1,...,N\}\backslash\{k\}$ for some $k$)
\item[--] add $r'$ to $\Omega$ provided the corresponding dual constraint is violated \\
(i.e. \eqref{acc_crit} holds).
\end{itemize}
These steps are repeated until a configuration $r'$ satisfying the acceptance criterion \eqref{acc_crit} has been found. When no such configurations have been found in sufficiently many steps, Lemma \ref{lem:justify_update} suggests that the current solution is optimal and the algorithm is terminated.

As already emphasized in the symmetric context in \cite{FriSchVoe-21}, the genetic learning aspect of the updating rule for $\Omega$ is that only ``successful'' configurations (i.e. ones where the current optimal plan is positive) are allowed to bear offspring. Moreover the dual state $u$ acts as ``dual critic'' of any proposed new configuration, accepting it only if it adds a new constraint which cuts off the current dual solution from the set of admissible potentials, in loose similarity to Wasserstein GANs \cite{arjovsky2017wasserstein}. Note also that the acceptance criterion is very cheap to check numerically, corresponding to just evaluating the cost and the dual solution at a single configuration. 

The updating step for $\Omega$ is completed by a tail-clearing procedure: one removes the oldest ``inactive'' configurations (i.e. $r\in\Omega\backslash\supp\gamma$) whenever $|\Omega|$ exceeds the a maximum allowed size 
\begin{equation} \label{maxsize}
   \beta \cdot(\ell_1+...+\ell_N).
\end{equation}
Here $\beta>1$ is a hyperparameter, chosen to be $3$ in all our numerical examples. The choice $\beta=1$ would be sufficient for exactness of the method (see Theorem \ref{thm:SparseOptimizers}), but not for efficient genetic learning. 
\vspace*{2mm}

\begin{figure}[http!]
    \centering
    \includegraphics[width=0.5\textwidth]{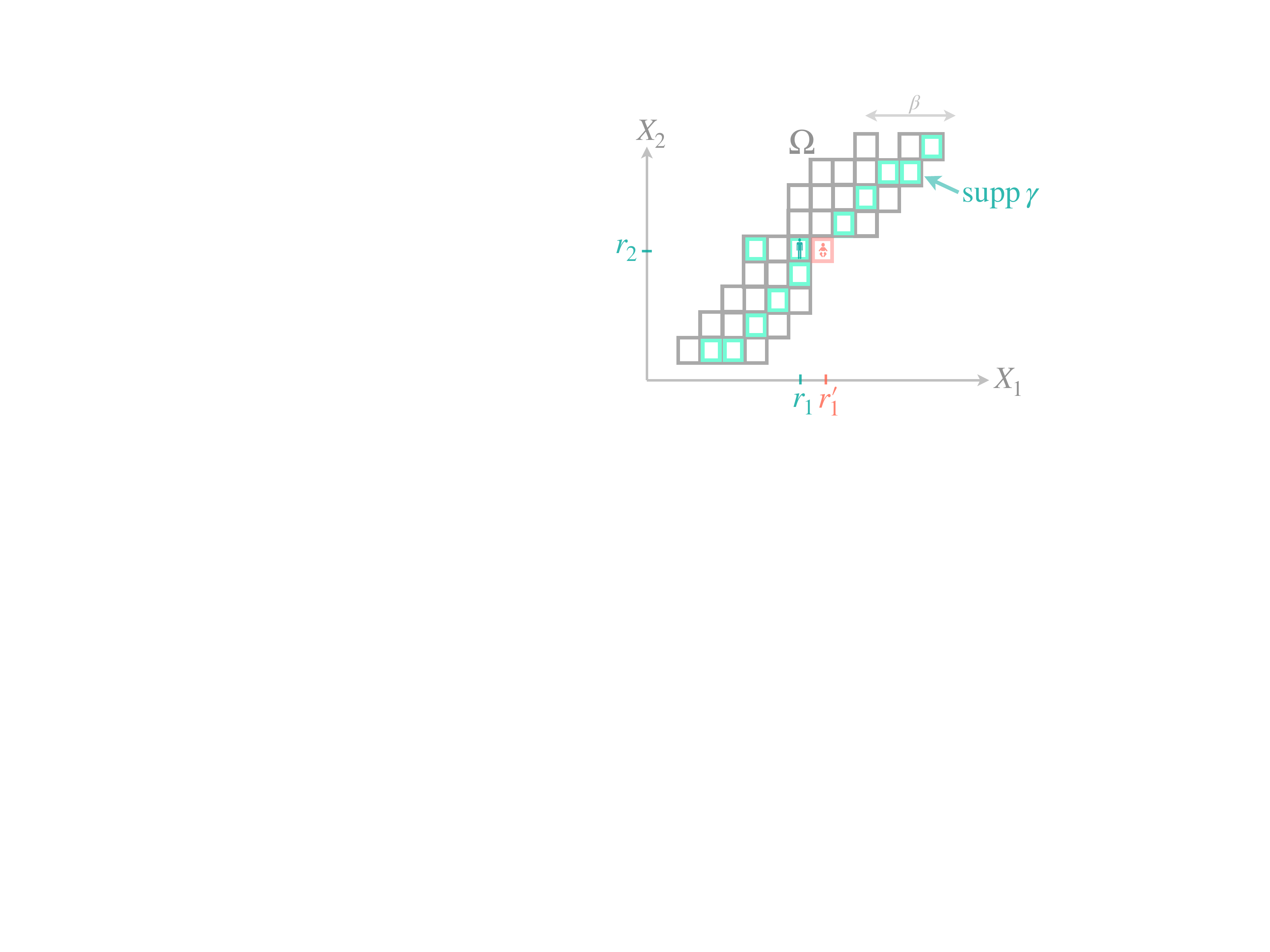} \vspace*{2mm}
    \begin{tikzpicture}[node distance=1.4cm ]
\scriptsize
\node (A) [block] {	Solve MMOT problem on $\Omega$\\
			$\gamma : \Omega \to [0,\infty)$ primal solution (Kant. plan)\\
			$u_i: X_i \to \mathbb{R}$ dual solutions (Kantorovich potentials)};
\node (B) [block, right=of A] { 	Pick random parent configuration $r=(r_1,...,r_N) \in \Omega$\\ with $\gamma(r) > 0$; create child $r'$ by\\ changing one entry of the configuration };
\node (C) [block, below=of B, yshift=0.4cm] { 	Accept child $r' \iff \sum\limits_{i=1}^Nu_i(r'_i) - c(r') > 0$ };
\node (D) [block, below=of A, yshift=0.4cm] {Update $\Omega^{new} = \Omega \cup \{r'\}$.\\ If $|\Omega^{new}| > \beta\cdot (\ell_1+...+\ell_N)$, remove oldest unused\\ configurations}; 
\node (E) [above= of A, yshift=-0.4cm] {Initialize reduced configuration set $\Omega$};

\draw [arrow] (A)--(B);
\draw [arrow] (B.south)++(-0.8,0)-- ++(0,-1);
\draw [arrow] (C.north)++(0.8,0) -- node[right] {if FALSE} ++(0,1);
\draw [arrow] (C)-- node[below] {if TRUE}(D);
\draw [arrow] (D)--(A);
\draw [arrow] (E)--(A);


\end{tikzpicture}
    \caption{Schematic description of the  GenCol algorithm. A single iteration solves the MMOT problem on a {\it reduced configuration set} $\Omega$ (gray squares), and then updates $\Omega$ based on the primal and dual MMOT solution. The updating rule for $\Omega$ is genetic: any proposed new configuration $(r_1',...,r_N')$ (red square) differs only in one entry from a `successful' configuration, i.e. one in the support of the optimal plan (green squares). The acceptance criterion is equivalent to the current dual solution no longer being a solution on $\Omega\cup\{r'\}$ (see text). The picture corresponds to the two-marginal case ($N\! =\! 2$).} 
    \label{fig:SchematicDescription}
\end{figure}    

For a summary of the algorithm in a flowchart respectively in pseudocode see Figure \ref{fig:SchematicDescription} and Algorithm \ref{algo:GenCol}.

\begin{algorithm}[http!]
\caption{Genetic Column Generation}	
\label{algo:GenCol}
\begin{algorithmic}[1]
\Require cost function $c$, marginals $\mu_k$, feasible reduced configuration set $\Omega$, hyperparameter $\beta \ge 2$
\While{not converged}
	\State $(\gamma,u) \gets$ solve \eqref{eq:RMMOT}, \eqref{eq:DRMMOT}  
	\Repeat
		\State Draw ``parent'' $r\in\supp\gamma$  
		\State Draw ``child'' $r' \in X_1\times ...\times X_N\backslash\Omega$ with $r'_i=r_i$ for all but one $i\in\{1,...,N\}$
	\Until{$\sum_iu_i(r'_i) - c(r') > 0$}
	\State Add new configuration $\Omega^{new} \gets \Omega \cup \{r'\}$
	\If{$|\Omega| > \beta \cdot (\ell_1+...+\ell_N)$} 
		\State remove oldest $\ell_1+...+\ell_N$ ``inactive'' configurations from $\Omega\backslash \supp \gamma $
	\EndIf
\EndWhile
\end{algorithmic}
\end{algorithm}

\vspace*{1mm}

{\bf Comparison with classical column generation.} Just like genetic column generation, classical column generation (translated into our context) alternates between solving the reduced problems \eqref{eq:RMMOT}--\eqref{eq:DRMMOT} and updating the set $\Omega$. (The name of the method comes from thinking about abstract linear programming in terms of the constraint matrix in the primal problem; adding a new configuration $r'$ to the reduced configuration set $\Omega$ in MMOT corresponds to adding a new column to the constraint matrix of the LP.) But the updating step is different: one adds the configuration $r'$ given by
\begin{equation} \label{pricing}
  r' = \underset{r}{\rm argmax} \Bigl( \sum_{i=1}^N u(r_i) - c(r)\Bigr).
\end{equation}
This problem is known as the pricing problem.
Moreover no tail-clearing is carried out, leading to possibly unresticted growth of the size of $\Omega$. 

Numerical tests and theoretical considerations show that the rather different (genetic) update in GenCol is essential for overcoming the curse of dimension. The classical rule \eqref{pricing}, or an unbiased random search for new configurations, or an unbiased random search  in a neighbhourhood of $\Omega$ would just turn the curse of dimension with respect to the state space into a curse of dimension with respect to the number of search steps. In fact, as shown in \cite{FriSchVoe-21}, even for symmetric MMOT with pairwise cost the pricing problem is NP-complete. 
\vspace*{2mm}

{\bf Comparison with the symmetric case.}  In symmetric MMOT, $X_1=...=X_N=:X$ (equal marginal spaces) and hence  $\ell_1=...=\ell_N=:\ell$, and $\mu_1=...=\mu_N=:\mu$ (equal marginals). The optimization in \eqref{discrete-MMOT-pb} is restricted to symmetric $\gamma \, : \, X^N\to\R$, i.e. $\gamma(r_{\sigma(1)},...,r_{\sigma(N)})=\gamma(r_1,...,r_N)$ for all permutations $\sigma$. Hence the marginal conditions can be reduced to a single one, $M_1\gamma=\mu$. Consequently exact optimizers $\gamma$ exist with even more sparsity by a factor $N$, where sparsity means number of nonzero coefficients of $\gamma$ with respect to the natural basis consisting of symmetrized Diracs. The latter are in one-to-one correspondence to the $N$-point configurations in the sector $(X^N)_{sym}=\{(a_{i_1},...,a_{i_N})\in X^N \, : \, 1\le i_1\le ... \le i_N\le \ell\}$, the reduced configuration set $\Omega$ becomes a subset of $X_{sym}$, and the maximum allowed size of $\Omega$ can be reduced by a factor $N$ from \eqref{maxsize} to 
$$
   \beta\cdot \ell.
$$
This makes GenCol even more efficient in the symmetric case. 
Another difference is that initialization of $\Omega$ is trivial in the symmetric case: the plan $\gamma_0=T_\sharp\mu$ with $T$ being the diagonal map $x\mapsto (x,...,x)$ from $X$ to $X^N$ is clearly feasible and its support (augmented by random configurations in $X_{sym}$) provides a feasible initial reduced configuration set.

\begin{proof}[Proof of Lemma \ref{lem:justify_update}]
    We have 
    \begin{align*}
         \langle c,\bar \gamma^\star \rangle &= \langle c, \gamma^\star \rangle_\Omega = \sum_{i=1}^N\langle p,u_i^\star \rangle
        = \max_{u}\left\{\sum_i\langle p_i,u_i\rangle : \sum_i u_i(r_i)\leq c(r) \forall r \in \Omega\right\} \\
        & \geq \max_{u}\left\{\sum_i\langle p_i,u_i \rangle : \sum_i u_i(r_i)\leq c(r) \forall r \in X_1\times\dots\times X_N\right\} \\
        & = \min_{\gamma}\Big\{ \langle c,\gamma\rangle: M_k\gamma=p_k \,\forall k, \, \gamma\ge 0  \Big\}.
    \end{align*}
    The first equality holds true because only zeros were added. The second and last equality hold true due to LP-duality. The third follows from the definition of $u^*$ and the inequality is a trivial consequence of  $\Omega \subset X_1\times \dots\times X_N$. But if $ \sum_{i=1}^Nu_i^\star(r_i) - c(r) \leq 0 \text{ for all } r \in X_1\times \dots\times X_N$, the inequality is clearly an equality and so $\bar\gamma^\star$ is a minimzer of the full MMOT problem.
\end{proof}
\subsection{Sparse initialization and multi-marginal north-west corner rule}
The GenCol algorithm requires one to find an initial sparse subset $\Omega \subset X_1 \times \dots X_N$ such that the reduced MMOT problem is feasible; the required size constraint is
\begin{equation} \label{sizeOmega}
   |\Omega| \le \beta \cdot (\ell_1+...+\ell_N).
\end{equation}
Here we introduce a simple and general method to generate such a set $\Omega$, and prove that it yields an initial feasible set $\Omega$ with less than $\sum_{k=1}^N \ell_k$ elements, hence fulfilling \eqref{sizeOmega} regardless of the choice of the hyperparameter $\beta>1$. (The set can then be augmented by arbitrary -- say, random -- configurations to saturate the bound \eqref{sizeOmega}.)  

Our method can be viewed as a generalization of the north-west corner rule to the multi-marginal case, or alternatively as an iterated monotone rearrangement (loosely similar to Knothe's transport). 

We begin by introducing an ordering of $X_k$,  $a_1^{(k)} < \dots < a_{\ell_k}^{(k)}$, to which we refer in terms of \textit{monotone} rearrangement. We start with an empty set $\Omega$ and add a first configuration $r = ( a_1^{(1)}, \dots, a_1^{(N)})$ and assign as much mass to this configuration as possible, given by $\gamma(r) = \min\{\mu_1(a_1^{(1)}), \dots, \mu_N(a_1^{(N)})\}$. For the next configuration we need to keep track of the remaining mass not allocated on all currently considered locations $a_1^{(k)}$. Therefore, we introduce
\[ b_k = a_1^{(k)} - \gamma(r). \]
There exists at least one $k \in \{1,\dots,N\}$ such that $b_k = 0$. This means that all mass at $a_1^{(k)}$ is assigned to the configuration $\gamma(r)$ and therefore we have to move to the next location $a_2^{(k)}$ and fill up $b_k$ with $\mu_k(a_2^{(k)})$. This leads to the next configuration $r = (a_i^{(1)}, \dots,a_i^{(N)})$, which is again added to the set $\Omega$. Then again, we assign as much mass as possible to it, from now on via 
\[ \gamma(r) = \min\{b_1,\dots,b_N\}, \]
compute the remaining mass, $b_k^{new} = b_k - \gamma(r)$, and update all locations $a_i^{(k)}$ in the current configuration $r$ for which $b_k =0$ to the next location $a_{i+1}^{(k)}$.
	
Iteratively, the generated configurations yield a sparse initial subset $\Omega$. The procedure is completely described in algorithm \ref{algo:InitOmega}.

\noindent
\begin{minipage}{0.5\textwidth}
\begin{algorithm}[H]
\caption{Multi-marginal north-west corner rule}	
\label{algo:InitOmega}
\begin{algorithmic}[1]
  \Require marginals $\mu_1,\dots,\mu_N$
  \State $\Omega = \emptyset$
  \State $r \gets (a_1^{(1)},\dots,a_1^{(N)})$
  \State $b \gets (\mu_1(r_1),\dots,\mu_N(r_N))$
  \While{$\exists k : r_k < a^{(k)}_{\ell_k}$}
    \State $\Omega \gets \Omega \cup \{r\}$
    \State $\gamma(r) \gets \min\{b_1,\dots,b_N\}$
    \For{k = 1:N}
      \State $b_k \gets b_k - \gamma(r)$
      \If{$b_k = 0$}
        \State $r_k \gets \min\{a_i^{(k)} > r_k\}$
        \State $b_k \gets \mu_k(r_k)$
      \EndIf
    \EndFor
  \EndWhile
  \State \Return Feasible set $\Omega$ and plan $\gamma$ 
\end{algorithmic}
\end{algorithm}    
\end{minipage}\hfill
\begin{minipage}{0.5\textwidth}
\begin{figure}[H]
    \includegraphics[width = \textwidth]{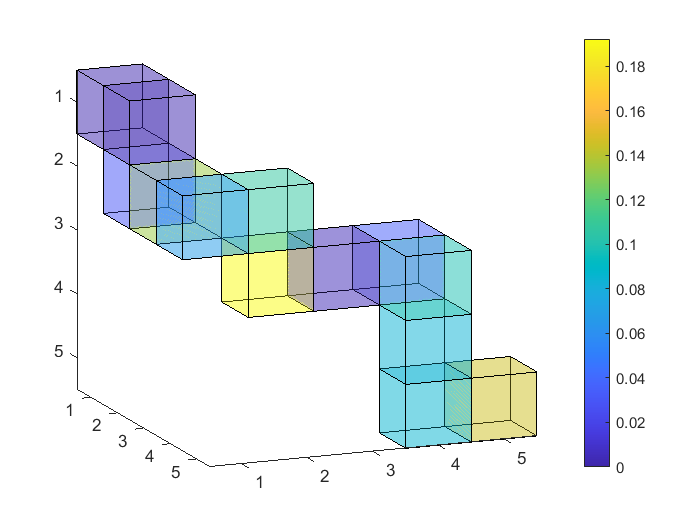}
    \caption{The NW corner rule was applied to find a feasible plan $\gamma$ for 3 marginals. $\Omega$ is its support.}
\end{figure}
\end{minipage}\\ [\baselineskip]

It is obvious that our algorithm yields a set $\Omega$ with less than $\sum_{k=1}^N \ell_k$ elements, because this number equals the total number of mass points contained in all the marginals and each new configuration $(r_1,...,r_N)$ added to $\Omega$ accounts completely for the remaining marginal mass at one of the $r_i$. A more precise count shows that the algorithm in fact exactly yields the amount of sparsity found in Theorem \ref{thm:SparseOptimizers}.

\begin{prp}
 Algorithm \ref{algo:InitOmega} yields a subset $\Omega \subset X_1 \times \dots \times X_N$ on which the multi-marginal problem \eqref{discrete-MMOT-pb} is feasible. The size of $\Omega$ is less or equal to $\sum_{k=1}^N (\ell_k-1) + 1 $.
\end{prp}

\begin{proof}
    Since $\gamma(r) =  \min\{b_1,\dots,b_N\}$, $b_k - \gamma(r) \geq 0$ for all $k = 1,...,N$ and there exists at least one $k$ such that equality holds. This ensures that at least one entry of the configuration $r$ is updated monotonically to the next element in $X_k$. Because we start with configuration $(a_1^{(1)},\dots,a_1^{(N)})$, every element can be updates at most $(\ell_k-1)$-times. Because we update the elements monotonically, this leads to at most $\sum_{k=1}^N (\ell_k -1) + 1$ configurations (if and only if exactly one element is updated in every iteration, e.g. for generic marginals).
    By construction, $b_k \gets \mu_k(r_k)$ and $b_k \gets b_k - \gamma(r)$ until $b_k = 0$, so the marginal constraints are fulfilled (i.e. $M_k\gamma = \mu_k$) and all mass will be distributed because $\int_{X_k}d\mu_k = 1$ for all $k =1,...N$.
\end{proof}

\section{Numerical results}
We implemented the algorithm in Matlab, where we used the Mosek toolbox \cite{mosek} because of its efficient LP solver. Moreover, Mosek -- unlike Matlab's inbuilt {\tt linprog} -- allows a hot start of the simplex algorithm in each interation, using the previous solution as initial state. Since the LP changes only slightly in each iteration as only one new variable is added, this leads to a further significant speed-up. 

We present four examples. The first is explanatory and illustrates how the algorithm works on a small two-marginal  example. In this case it would of course not be necessary to use a method that enforces temporary plans to remain sparse, but nevertheless the algorithm works just fine, converging exponentially at a steady rate to the exact solution. Our other three examples are large: we compute Wasserstein barycenters of MNIST digits (representing accurate summaries of these digits), weighted Wasserstein barycenters of shapes (representing morphed shapes), and cubic spline interpolations in Wasserstein space. All examples were simulated in Matlab R2022a.

\subsection{Explanatory Example} \label{sec:small}
To illustrate the way the algorithm works we start with a small example (which, of course, could be solved with many other methods not requiring sparsity of plans in each update). Consider $N=2$ marginals, supported on a uniform mesh with $\ell = 100$ gridpoints in the interval  $[0,1]$ (Figure \ref{fig:OneDimMs}), and the standard quadratic cost $c(x,y)=|x-y|^2$.
\begin{figure}
    \centering
    \includegraphics[width=0.4\textwidth]{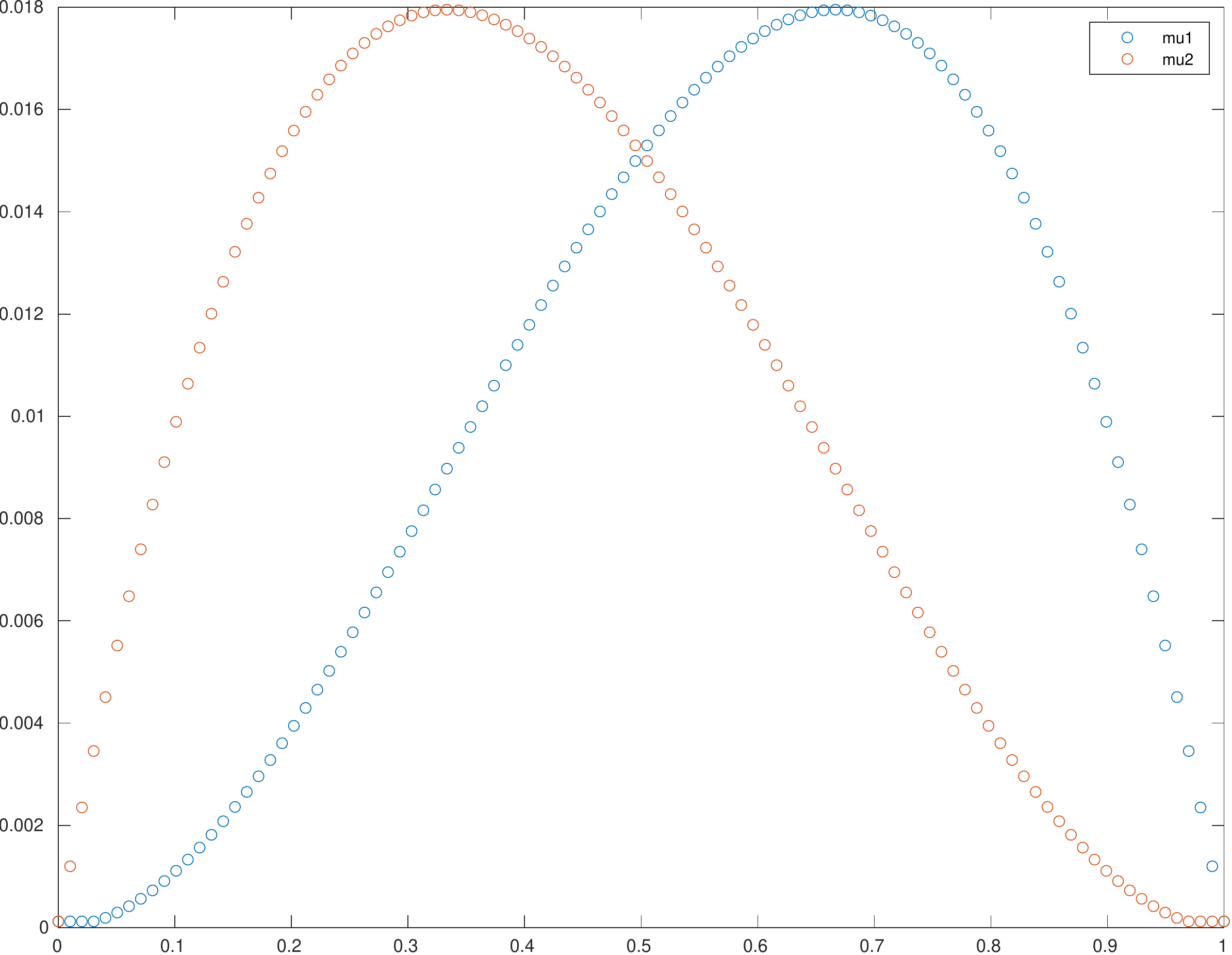}
    \caption{The two marginals in Example \ref{sec:small}}
    \label{fig:OneDimMs}
\end{figure}
We chose $\mu_2$ to be the reflection $\mu_2(x) = \mu_1(1-x)$. A simple and obviously non-optimal transport plan belonging to $\Pi(\mu_1,\mu_2)$ is $\gamma_0:=T_\sharp \mu_1$ with $T(x)=(x,1-x)$. Its support (augmented with random configurations) provides us with a feasible initial reduced configuration set $\Omega$. (Using the initial configuration from the NW corner rule is not instructive in this example as it would already provide the support of the optimal plan, and  GenCol would find this plan in one iteration.) Starting from this set, 
GenCol finds the optimal plan to machine precision after less than 900 iterations. The evolution of optimal plan, Kantorovich potentials, and cost during the iteration is presented in Figures \ref{fig:OneDim} and \ref{fig:OneDimCost}.
\begin{figure} 
    \centering
    \includegraphics[width=\textwidth]{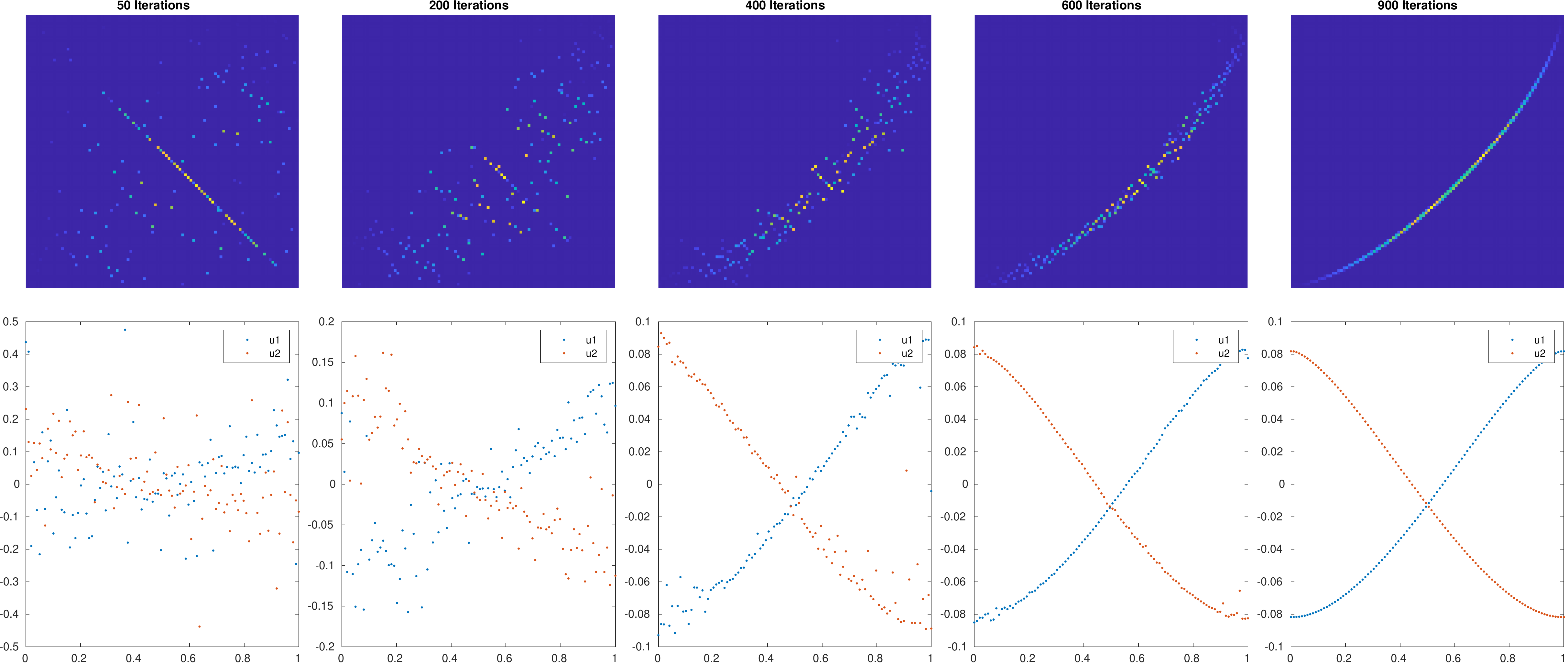}\\ [\baselineskip]
    \caption{Top row: evolution of the optimal transport plan under the GenCol algorithm. Bottom row: evolution of the Kantorovich potentials. GenCol produces a sparse solution in every iteration, converging to the exact discrete solution. The plans and potentials successively ``learn'' from each other, as in adversarial learning.}
    \label{fig:OneDim}
\end{figure}
\begin{figure}
    \centering
    \includegraphics[width = 0.6\textwidth]{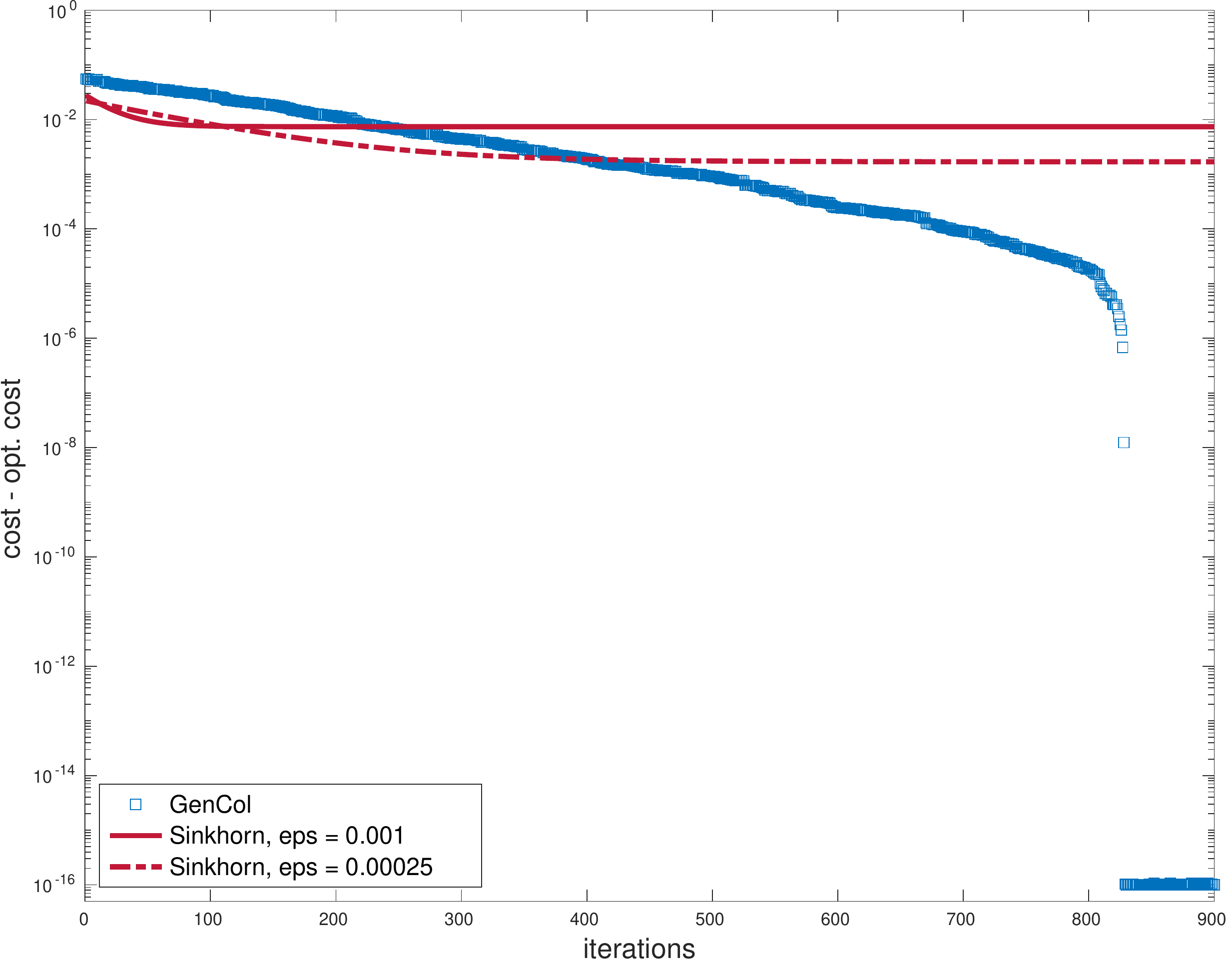}
    \caption{Convergence of GenCol in terms of iterations, for the two-marginal example in Figure \ref{fig:OneDim}. The plot reveals exponential convergence at a steady rate (blue squares). For few-marginal problems one can also use the Sinkhorn algorithm, whose performance is shown for comparison. If low accuracy ($\sim 10^{-2}$) is sufficient, Sinkhorn is much faster, and hence the recommended method (solid red line). For moderate accuracy ($\sim 10^{-3}$), the Sinkhorn parameter needs to be set to a lower value and both methods become comparable, requiring $\sim 400$ iterations. After convergence, GenCol delivers the exact solution and Sinkhorn a regularization thereof.
    }
    \label{fig:OneDimCost}
\end{figure}
One can see how improved configurations are found and accepted. First, the Kantorovich potentials are updated to account for violated constraints. The optimal plan then follows by occupying cheaper configurations. The sparsity of the optimal plan is kept in every iteration.

\subsection{MNIST handwritten digits}
Next we computed the mesh-free Wasserstein barycenter of MNIST handwritten digits. 
The MNIST data set \cite{MNIST} is a well known benchmark data set in machine learning consisting of a collection of  handwritten digits (0 to 9), of the size of 28$\times$28 pixels.

We used 10 marginal images, which were drawn at random from the training set. The MMOT problem then has $28^{20} \approx 8.7\cdot10^{28}$ variables and $10\cdot 28^2 \approx 8000$ constraints. Hence by the theory from section \ref{sec:sparse} the barycenter is a superposition of only $\approx 8000$ Dirac measures. From a machine learning perspective, the map from the training set to the collection of Wasserstein barycenters can be viewed as a feature map yielding an efficient summary of main features of each digit. The usefulness of this feature map in classification problems, especially in reducing the need for large training sets, will be discussed elsewhere.




Figure \ref{fig:MNISTallBC} shows the accurate mesh-free barycenter for all digits, obtained by applying the GenCol algorithm to the multi-marginal problem \eqref{MMOT-pb}, \eqref{barycost}. As explained in the Introduction, the mesh-free barycenter is supported on an $N^2$ times finer grid than that of the marginals. The fixed-mesh barycenter, computed by applying the Sinkhorn algorithm to the coupled two-marginal formulation, is shown for comparison. 

We do not claim that the mesh-free GenCol barycenter is always superior to the fixed-mesh Sinkhorn barycenter. The latter is superior in speed in the present example. By contrast the former yields the true barycenter rather than a regularization, and is superior in resolution. 
It hence consitutes a viable alternative method, especially in high-accuracy applications or for input data with bad resolution. 

\begin{figure}
    \centering
    \includegraphics[width = 0.8\textwidth]{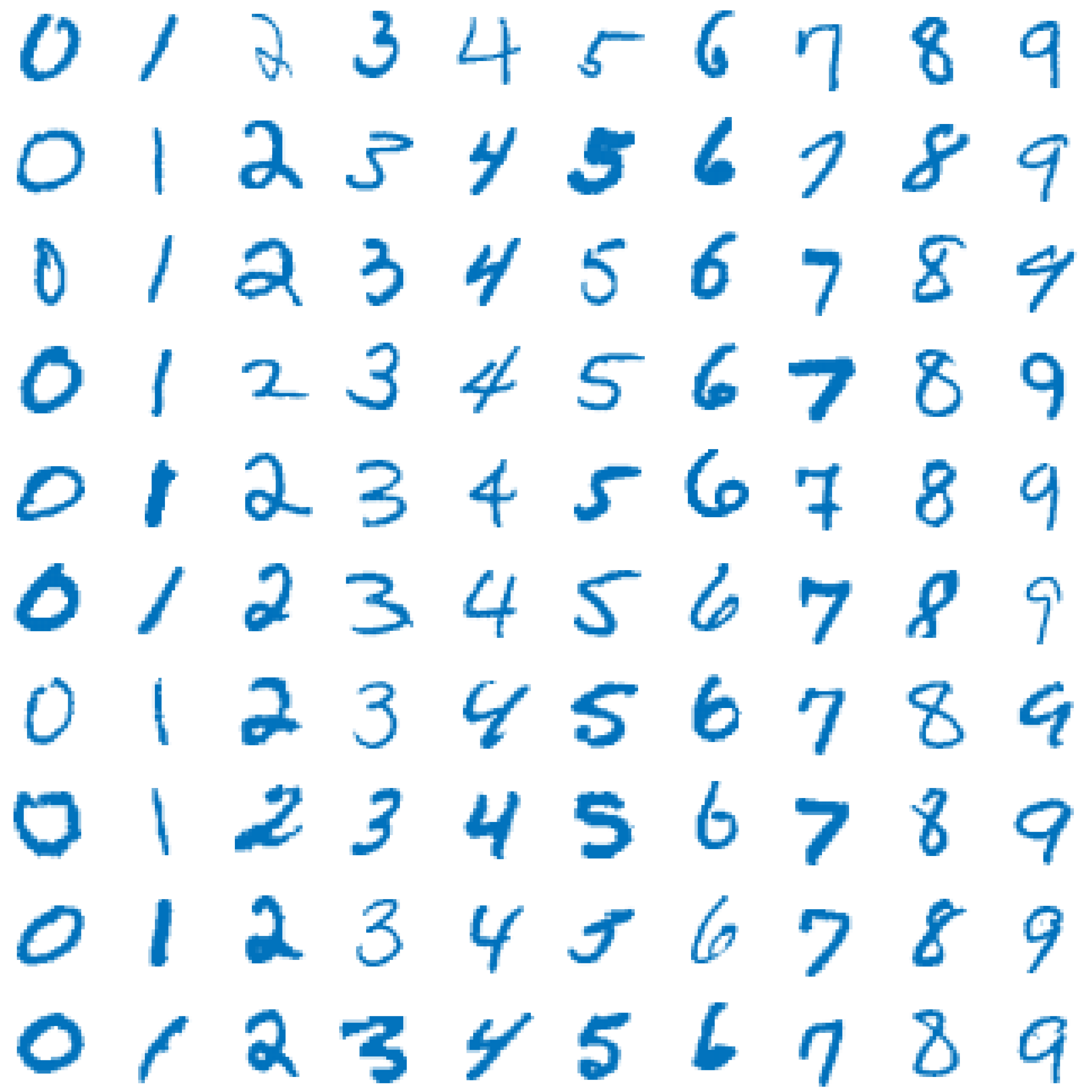}\\[-\baselineskip]
    \rule{0.8\textwidth}{.4pt}
    \vspace*{4pt}
    \includegraphics[width = 0.8\textwidth]{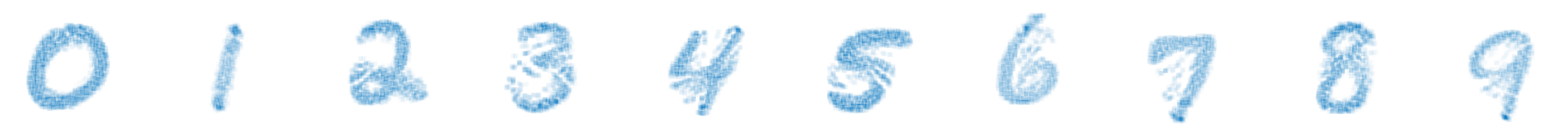}\\[-\baselineskip]
    \rule{0.8\textwidth}{.4pt}
    \vspace*{4pt}
    \includegraphics[width = 0.8\textwidth]{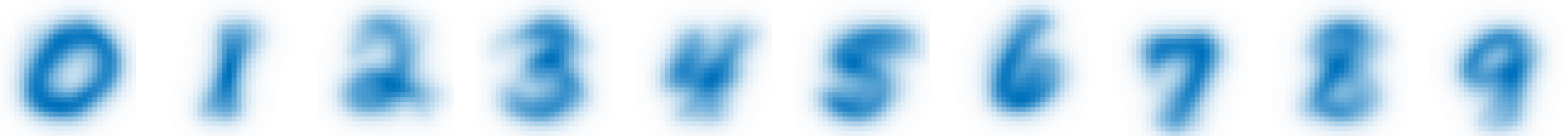}
    \caption{For each digit we drew 10 MNIST images at random and computed their barycenter. The mesh-free barycenters computed with 
    GenCol are presented in the second to last row. The last row shows the coupled-two-marginal barycenter computed with the Sinkhorn algorithm with $\varepsilon = 0.05$.}
    \label{fig:MNISTallBC}
\end{figure}

\subsection{Shape morphing}
In our next example we computed high-resolution weighted Wasserstein barycenters of three shapes to perform shape morphing. The input shapes had resolution $256 \times 256$, yielding $2.8 \cdot 10^{14}$ variables in the multi-marginal problem. The result is shown in Figure \ref{fig:Shapes3}. 

The collection of weighted Wasserstein barycenters computed with GenCol applied to the multi-marginal formulation accurately resolves the true trajectory of each pixel in continuous space, as explained in the Introduction. Shapes with sharp boundaries can be read off via smoothing and thresholding, see Figure \ref{fig:Shapes3_level}. 
%
%
%

\begin{figure}
    \centering
    \includegraphics[width = 0.8\textwidth]{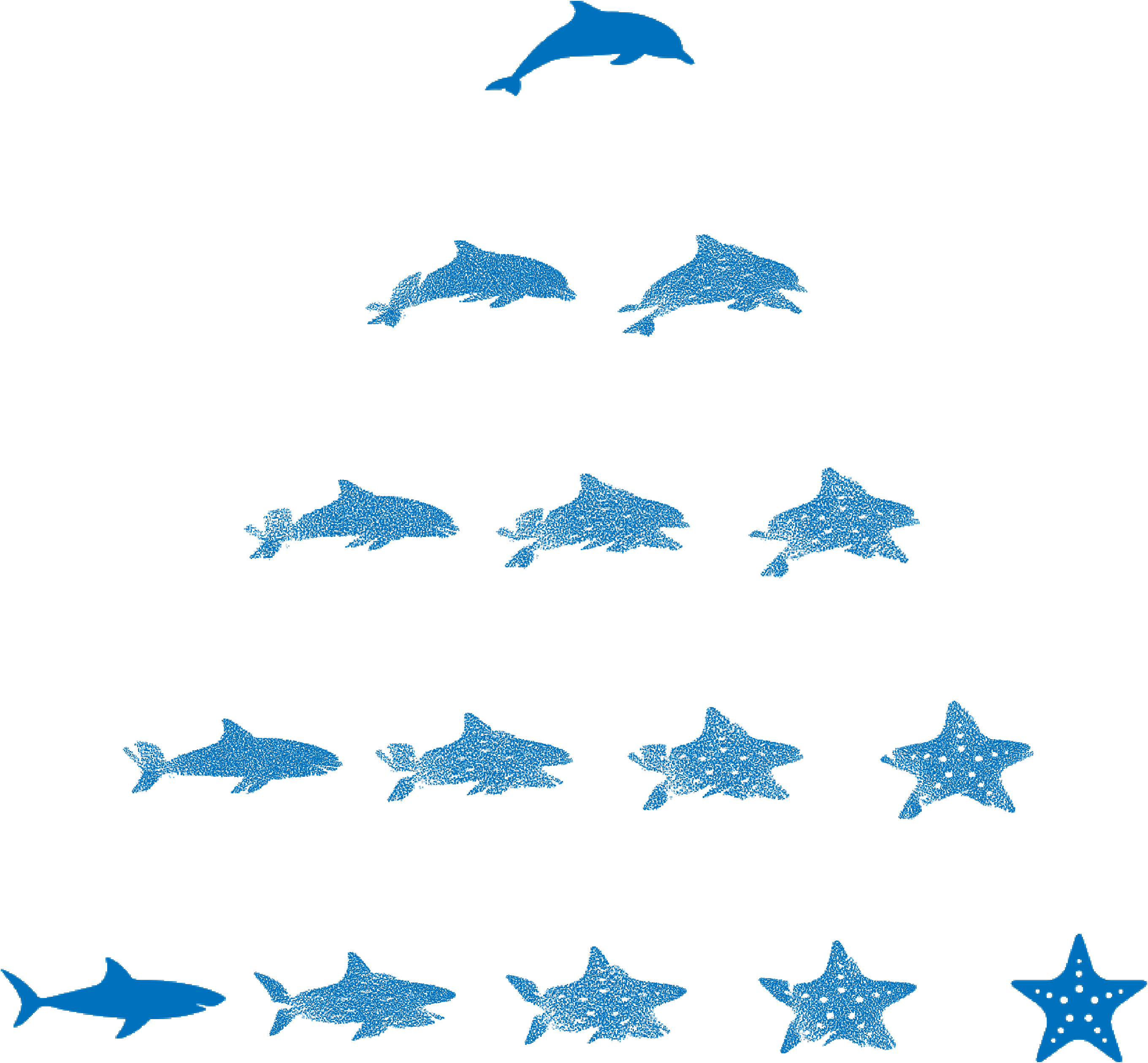}
    \caption{Accurate weighted mesh-free Wasserstein barycenters of three shapes computed with GenCol. One can track the trajectory of each pixel. The edges are the Wasserstein geodesics between two shapes.}
    \label{fig:Shapes3}
\end{figure}
\begin{figure}
    \centering
    \includegraphics[width = 0.8\textwidth]{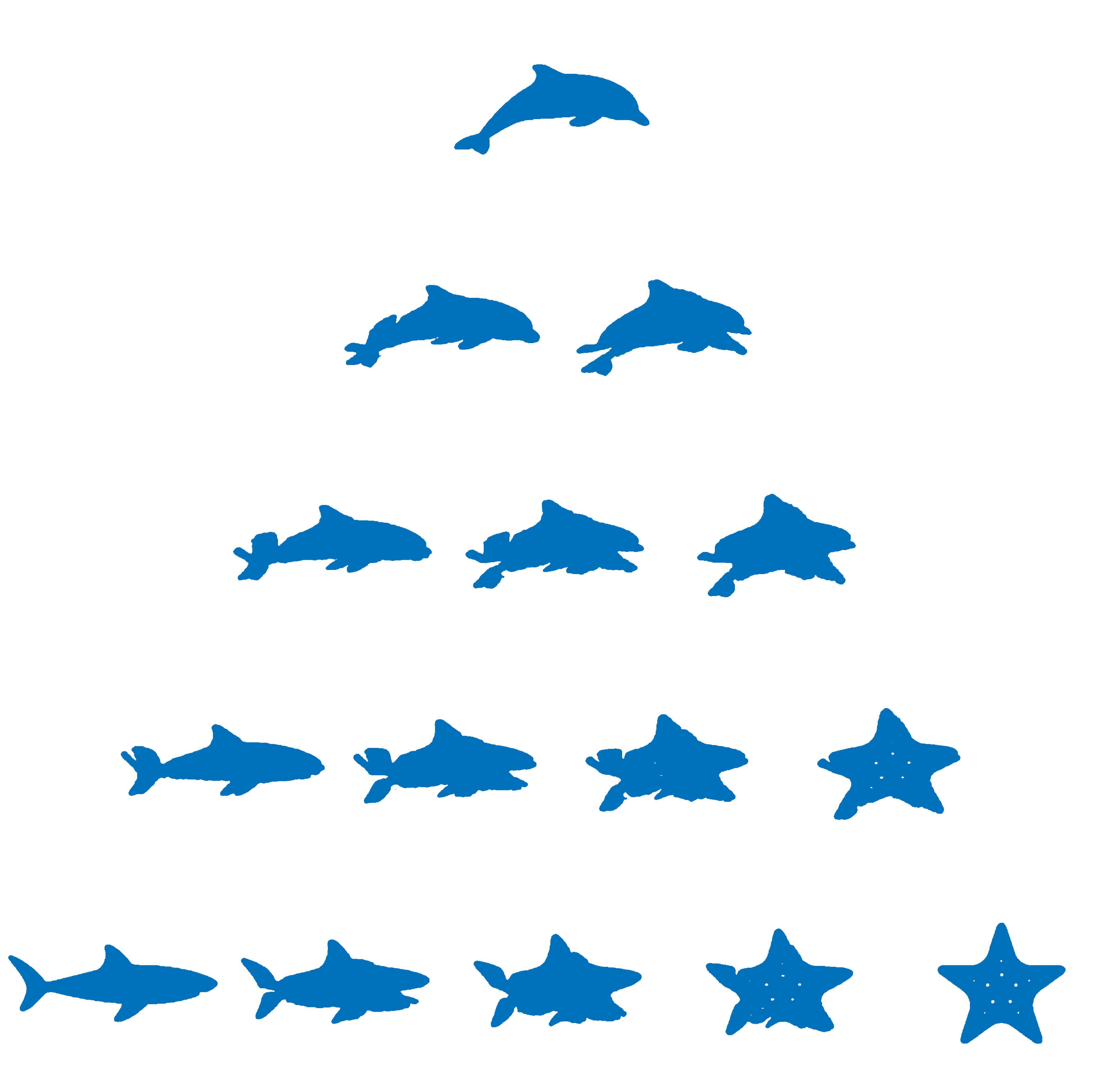}
    \caption{Post-processing the high-resolution images from Figure \ref{fig:Shapes3} via smoothing and thresholding produces shapes with 
    sharp boundaries. 
    Since the smoothing is added a posteriori, 
    it can easily be altered and optimized application specifically.}
    \label{fig:Shapes3_level}
\end{figure}

\subsection{Smooth interpolation in Wasserstein space}
In our final example we computed higher order Wasserstein interpolations of given time series of probability measures on $\R$ and $\R^2$. The governing multi-marginal problem is given by eqs.~\eqref{MMOT-pb} and \eqref{splinecost} and the interpolating path of measures is given by eq.~\eqref{interpol}. For simplicity we used the approximate cost $\tilde{c}$ given in the Introduction. We used, respectively, 6 Gaussian marginals on 101 gridpoints and 5 Gaussian marginals on a 50 $\times$ 50 grid. The multi-marginal problem has $10^{12}$ respectively $10^{17}$ variables. The interpolations computed with GenCol are shown in Figure \ref{fig:WSplines}. 

Unlike first-order interpolation,  which produeces piecewise linear particle trajectories due to the fact that 
$$
  \mu_t = \bigl( (1-\lambda)id + \lambda T\bigr)_\sharp \mu_{t_i}
$$
where $t=(1-\lambda)t_i+\lambda t_{i+1}$ and $T$ is the optimal map from $\mu_{t_i}$ to $\mu_{t_{i+1}}$, we see that spline interpolation produces smooth particle trajectories. 

In contrast to the Wasserstein barycenter problem, the smooth interpolation problem does not admit a coupled two-marginal formulation,  underlining
the need for a many-marginal solver as presented here. 

\begin{figure}
    \centering
    \includegraphics[width = 0.46\textwidth]{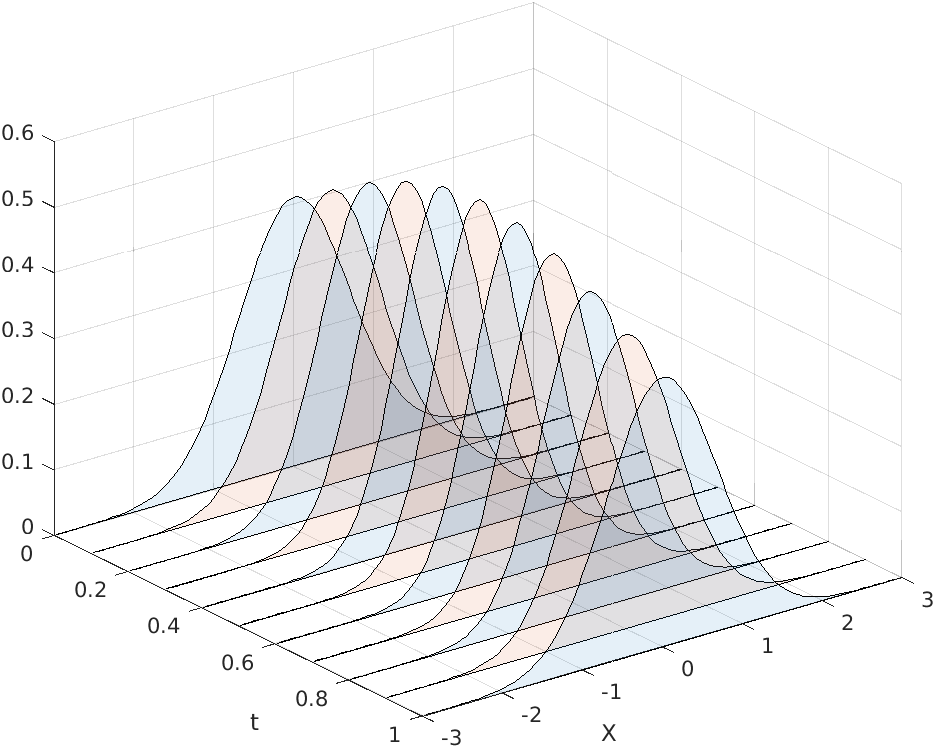} \hfill{}\includegraphics[width = 0.46\textwidth]{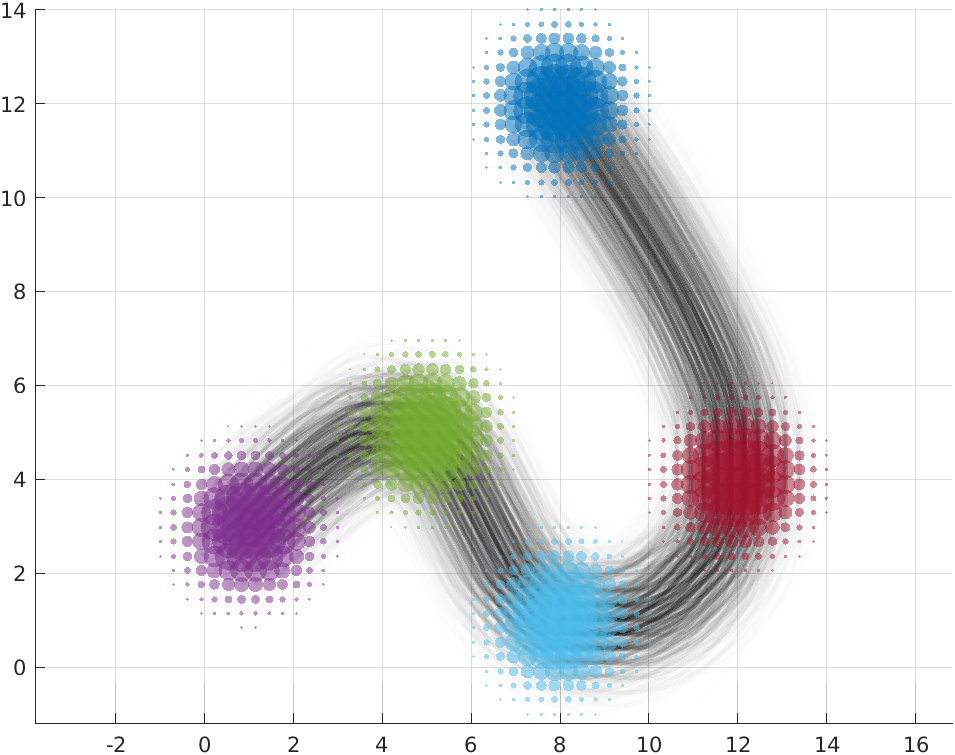}
    \caption{Wasserstein cubic splines computed with GenCol. Left: The six blue 1-D Gaussians represent the marginals. The red distributions are interpolations, visualized by applying a kernel to the discrete solution. Right: Wasserstein cubic spline through five 2-D Gaussian distributions. The lines indicate particle trajectories and the line darkness indicates the amount of mass transported.}
    \label{fig:WSplines}
\end{figure}

\section{Conclusions and outlook}
In all our examples the GenCol algorithm performed accurately and efficiently, allowing to treat multi-marginal optimal transport problems with up to $10^{30}$ unknowns with Matlab on a laptop. 

Key tasks for future work include a theoretical analysis of the observed accuracy and efficiency, and parallelization and implementation on high-performance hardware for  large-scale applications.

\printbibliography

\clearpage
\appendix
\section{Dubins' theorem for convex polytopes}\label{apdx:ExtremePoints}
Here we present a simple proof of the following result:
	
\begin{thm*}[Dubins' theorem for polytopes] 
Let $L$ be the intersection of a bounded convex polytope $K\subset\R^d$ with $n$ hyperplanes. Then every extreme point of $L$ is a convex combination of at most $n + 1$ extreme points of $K$.
\end{thm*}
Dubins established this result for general closed and bounded convex sets \cite{dubins-62}. 
But for polytopes the result is much easier to prove because we can utilize their partially ordered set of faces, so we included a proof here for lack of a suitable reference.
	
Recall some convex geometry. A point $p$ in a convex set  $K\subset\R^d$ is called an \textit{extreme point} of $K$ if it cannot be written as convex combination of other points $x,y \in K$. Denote the set of extreme points by $\mathcal{E}(K)$. 

A convex \textit{polytope} can either be defined as the intersection of finitely many closed half spaces or as the convex hull of finitely many points. If one additionally assumes the polytope to be bounded, the two definitions are equivalent and from now on we stick to the second definition. 
	
If all the points are convex independent, i.e. none of them is a convex combination of the others, they form the set of vertices of the polytope, which coincides with the  set of extreme points.
	
A prototypical example of a polytope, and the one to which we applied the above theorem in section \ref{sec:sparse}, is the set of probability measures on a discrete space $X, |X| = \ell$. Its extreme points are the Dirac measures at each point of $X$. Every convex combination of those measures is again a measure on $X$ and every measure on $X$ can be written as convex combination of such Dirac measures. The measures can be identified with the probability vectors of length $\ell$ whose components represent the amount of mass at every point. In its vector representation the polytope is therefore spanned by the unit vectors in $\R^\ell$ and called \textit{probability simplex}.
	
A face of a convex polytope is any non-empty intersection of the polytope with a half space such that none of the (relative) interior points of the polytope lie on the boundary of the half space. The dimension of a face is defined to be the dimension of the smallest affine subspace containing it. The faces of a polytope form a partially ordered set, or poset, called \textit{face lattice}, where the ordering is given by set inclusion. The smallest faces are the vertices (0-dimensional faces) and the largest face  (containing all faces) is the polytope itself. Each face is itself a bounded convex polytope whose vertices or extreme points form a subset of the vertices or extreme points of $K$.
	
With those tools we are able to prove the above theorem.
\begin{proof}[Proof of the theorem]
First we show that $L$ is again a bounded convex polytope. Boundedness is clear, as is convexity (since the intersection of convex sets is again convex). Remember that a polytope can be defined as the intersection of finitely many closed half spaces. Any hyperplane $H = \{x  :  \langle a,x\rangle = d\}$ is the intersection of the half spaces $\langle a,x\rangle \geq d$ and $\langle a,x\rangle \leq d$. If one adds these two half spaces to the half-space representation of $K$, the resulting set is still the intersection of finitely many half-spaces and therefore again a polytope.
		
The statement about extreme points can be shown by using the face lattice of the polytope. The extreme points of $K$ are the 0-dimensional faces (vertices) of the polytope. All extreme points $p\in \mathcal{E}(K)$ not excluded by the intersection with the hyperplanes $H_1,...,H_n$ (i.e. $p \in H_1\cap ... \cap H_n$) remain in the set of extreme points $p\in \mathcal{E}(L)$, because they still cannot be written as convex combination of elements of the smaller set $L$.  
Conversely, let $p \in (H_1\cap ... \cap H_n)\cap K$. Let $F$ be the minimal face of $K$ containing $p$ and let $d'$ be its dimension. Due to the minimality of the face, $p$ lies in its relative interior, i.e. there exists an open $d'$-dimensional neighborhood around $p$ in $F$. The intersection of this neighborhood with $H_1\cap ... \cap H_n$ can reduce its dimensionality by at most $n$, so $p$ can only be an extreme point of $L$ if it lies in a face of dimension at most $n$ of $K$. But by Carath\'{e}odory's theorem, every element of an $n$-dimensional face is a convex combination of at most $n+1$ extreme points of the face. Since the latter points belong to $\mathcal{E}(K)$, the assertion follows.
	\end{proof}
\section{Wasserstein cubic splines cost function} \label{apdx:CubicSplineCost}
The Kantorovich formulation of the variational Wasserstein cubic spline problem seeks the optimal distribution $\mu$ on the space ${\mathcal H}$ of $C^2$-paths $x: [0,1] \to X$ which minimizes
\[\int_{{\mathcal H}} c(x) d\mu(x) \text{ subject to } (e_{t_i})_\#\mu = \mu_i \, \forall i=0,1,...,N.\]
Here $X$ is a subset of $\R^d$, the cost $c(x)$ is the spline energy $\int_0^1|\Ddot{x}(t)|^2dt$, and $e_t(x)=x(t)$ is the evaluation map at time $t$ so that $t\mapsto (e_t)_\sharp\mu$ describes a path of probability measures in $\calP(X)$. The marginal conditions prescribe this path at finitely many time points $t_i \in [0,1]$. As shown in \cite{benamou2019second}, any optimizer is supported on the set of cubic splines determined by $\{(t_i,x(t_i))\}$. Hence the problem can be reduced to solving the MMOT problem of minimizing 
\begin{equation*}
    \int_{X^{N+1}} c(x_0,x_1,\dots,x_N) \,d\gamma(x_0,x_1,\dots,x_N) 
    \text{ subject to } M_k\gamma = \mu_k  \forall k=0,1,...,N,
\end{equation*}
with $c$ given by \eqref{splinecost}. The interpolant at an arbitrary time $t\in[0,1]$ is then $(E_t)_\sharp \gamma$, where $E_t(x_0,...,x_N)=x(t)$ is the value of the optimal path in \eqref{splinecost} at time $t$.
The cost function $c$ still needs to be computed. The spline can be defined in the intervals $[t_j,t_{j+1}]$ via its second derivative
\[ \Ddot{x}(t) = M_j\frac{t_{j+1} - t}{h_{j+1}} + M_{j+1}\frac{t-t_j}{h_{j+1}}, \]
where $h_{j+1} := t_{j+1}-t_j$ and the $M_j$ are solutions of a linear system specified below. Let 
\[\lambda_0 = d_0 = 0, \mu_N = d_N = 0, \]
and for $j = 1,\dots,{N-1}$ define
\begin{gather*}
    \lambda_j := \frac{h_{j+1}}{h_j+h_{j+1}},\quad \mu_j := 1-\lambda_j,\\ 
    d_j := \frac{6}{h_j+h_{j+1}}\left(\frac{x_{j+1} - x_j}{h_{j+1}} - \frac{x_j - x_{j-1}}{h_j}\right).
\end{gather*}
Then the coefficients $M_j$ are the solutions of
\begin{equation*}
    \begin{pmatrix}
        2 & \lambda_0 & & & & 0\\
        \mu_1 & 2 & \lambda_1 & & & \\
        & \mu_2 & \cdot & \cdot & & \\
        & & \cdot & \cdot & \cdot & \\
        & & & \cdot & 2 & \lambda_{N-1} \\
        0 & & & & \mu_N & 2
    \end{pmatrix}
    \begin{pmatrix}
        M_0 \\ M_1 \\ \cdot \\ \cdot \\ \cdot \\ M_N
    \end{pmatrix}
    =
   \begin{pmatrix}
        d_0 \\ d_1 \\ \cdot \\ \cdot \\ \cdot \\ d_N
    \end{pmatrix}.
\end{equation*}
Given the coefficients $M_j$, the cost $c(x)$ is the sum of the integrals
\begin{align*}
    \int_{t_{j}}^{t_{j+1}} |\Ddot{x}(t)|^2\,dt &= \frac{1}{h_{j+1}^2}\int_{t_{j}}^{t_{j+1}} |\underbrace{(M_{j+1} - M_j)}_{=:a}t + \underbrace{(M_jt_{j+1} - M_{j+1}t_j)}_{=:b}|^2\,dt\\
    &= \frac{1}{h_{j+1}^2}\left[ \frac{a^2}{3}t^3 + abt^2 + b^2t \right]\Bigg|_{t_j}^{t_{j+1}}\\
    &= \frac{1}{h_{j+1}^2}\left[\frac{a^3}{3}(t_{j+1}^3-t_j^3) + ab(t_{j+1}^2-t_j^2)+b^2(t_{j+1}-t_j)\right]\\
    &= \frac{1}{h_{j+1}}\left[\frac{a^3}{3}(t_{j+1}^2+t_{j+1}t_j+t_j^2) + ab(t_{j+1}+t_j)+b^2\right].
\end{align*}
For equidistant time steps, this cost is well approximated by the simple expression given in the Introduction. 

\end{document}